\documentclass{amsart}

\usepackage{amsfonts,amssymb,amscd,amsmath,enumerate,verbatim,newlfont,calc}

\usepackage{amsfonts}
\usepackage{graphicx}

\newtheorem{theorem}{Theorem}[section]

\newtheorem{cor}[theorem]{Corollary}

\newtheorem{definition}[theorem]{Definition}

\newtheorem{lemma}[theorem]{Lemma}

\newtheorem{prop}[theorem]{Proposition}
\newtheorem{remark}[theorem]{Remark}

\begin{document}

\title[A Piecewise Contractive Map on Triangles]{A Piecewise Contractive Map on Triangles}
\author[Samuel Everett]{Samuel Everett}

\begin{abstract}
We study the dynamics of a piecewise map defined on the set of three pairwise nonparallel, nonconcurrent lines in $\mathbb{R}^2$.  The geometric map of study may be analogized to the billiard map with a different reflection rule so that each iteration is a contraction over the space, thereby providing asymptotic behavior of interest.  Our study emphasizes the behavior of periodic orbits generated by the map, with description of their geometry and bifurcation behavior.  We establish that for any initial point in the space, the orbit will converge to a fixed point or periodic orbit, and we demonstrate that there exists an infinite variety of periodic orbits the orbits may converge to, dependent on the parameters of the underlying space.
\end{abstract}

\maketitle

\graphicspath{ {Diagrams/} }

\section{Introduction}

In \cite{wang}, Wang and Zhang showed that particular classes of iterated function systems can provide robust dynamics when defined on the boundary of particular triangular plane regions.  From such work, it is natural to ask if different systems may provide similarly interesting asymptotic behavior when defined over boundaries of triangular regions.  

The purpose of this note is to study the asymptotic properties of orbits generated by a piecewise linear map, defined on the space of three pairwise nonparallel, nonconcurrent lines in $\mathbb{R}^2$.  In many cases, we find that such dynamics reduce to simply studying the properties of the orbits contained in Euclidean triangles.  

Our main result demonstrates how iteration of the map defined here is asymptotically periodic in the space, or will ``degenerate" to a fixed point.  We show that the dynamics of the map are not determined by the parameters defining the map, but rather the parameters of the intersecting lines composing the space.

This paper is organized as follows.  Section \ref{sec1} introduces the fundamental geometric properties and dynamics of the system.  Section \ref{sec2} shows that for every set of three pairwise nonparallel, nonconcurrent lines in $\mathbb{R}^2$, the map of study will converge to a periodic orbit or fixed point.  Section \ref{sec3} focuses on the dynamics of the map restricted to boundaries of isosceles triangles, with discussion of bifurcation points and associated bifurcation diagram.  

\subsection{Acknowledgments}
The author is grateful to Professor Nicolai Haydn (USC) for guiding me and this project along during my first year of university, and more generally for the incredible mentorship and for teaching me a great deal of mathematics.

\section{Definitions and basic properties}\label{sec1}

Let $L_1, L_2, L_3$ be any set of pairwise nonparallel, nonconcurrent lines in $\mathbb{R}^2$, and let $X = L_1 \cup L_2 \cup L_3$.  As such, we can put $\{L_1 \cap L_2, L_1 \cap L_3, L_2 \cap L_3\} = \{A, B, C\}$ with $A, B, C$ distinct, so that $\bigtriangleup ABC$ is a Euclidean triangle. Call this triangle $Q$, and let $\partial Q$ denote the boundary of the triangle so that $\partial Q \subset X$.

\begin{definition}\label{def1}
Let $T:X \rightarrow X$ be defined so that, if $x_0 \in L_i$ for $i=1, 2, 3$, and $x_1$ denotes the perpendicular projection of $x_0$ onto line $L_j$, and $x_2$ denotes the perpendicular projection of $x_0$ onto line $L_k$ for $i, j, k$ distinct, then $T(x_0) \in \{x_0, x_1, x_2\}$ such that if $d(x_1, x_0) > d(x_2 , x_0)$ then $T(x_0) = x_1$, and if  $d(x_1, x_0) = d(x_2, x_0)$, then $T(x_0) = x_0$.
\end{definition}

Figure \ref{fig1} offers illustration of what iteration of the map $T$ looks like in a triangle, visually demonstrating the kind of behavior this paper explores.  We let $T^n$ denote the $n$th iteration of $T$, such that $T^1(x) = T(x)$, and $T^n(x) = T(T^{n-1}(x))$. For any $x \in X$, we let $\mathcal{O}(x)$ denote the \emph{orbit of x} under $T$, so that $\mathcal{O}(x) = \{x, T(x), T^2(x), T^3(x), ...\}$ is the set of points obtained by iterating $T$ with initial value $x$.  Furthermore, if $T(x) = x$, we call $x$ a \emph{fixed point}.  

Let $\Omega$ denote the set of all fixed points and all preimages of $T$ from fixed points in $X$.  That is, if $h \in X$ is a fixed point, then $\{h, T^{-1}(h), T^{-2}(h), T^{-3}(h), ...\} \subset \Omega$.  We label $X' = X \setminus \Omega \cup \{A, B, C\}$ for line intersection points $A, B, C$. 

\begin{figure}
\centering
\begin{minipage}[t]{5cm}
    \hspace*{.4cm}
	\includegraphics[scale=.26]{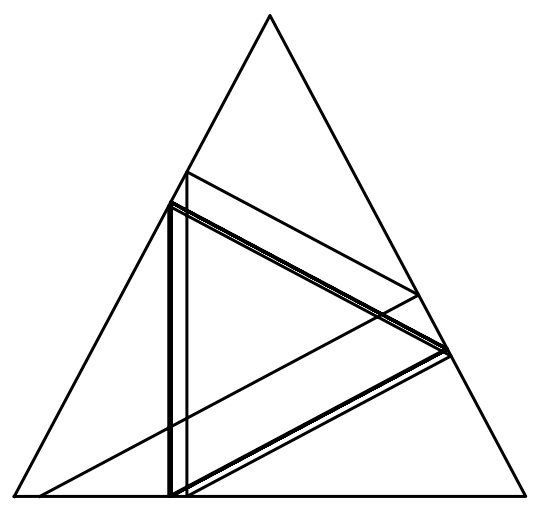}
\end{minipage}
\hspace{3cm}
\begin{minipage}[t]{4cm}
    \hspace*{0cm}
	\centering
	\includegraphics[scale=.28]{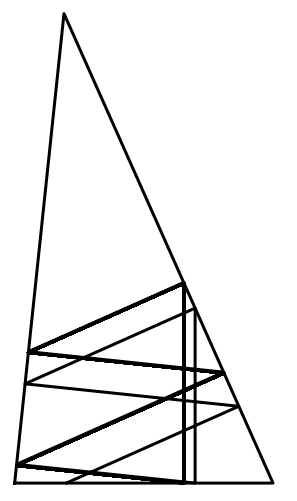}
\end{minipage}
\caption{The result of 100 iterations of the map $T$ in two different triangles}
\label{fig1}
\end{figure}

\begin{definition}
For interior angle $\theta$ in $Q$, we call point $H$ a \emph{bisector point}, if $H$ is the point of intersection between the angle bisector of $\theta$ and the edge opposite $\theta$.
\end{definition}

It is clear that all bisector points are fixed points under $T$.  Note that the set of bisector points need not constitute the entire set of fixed points however: there may exist two other fixed points for each line, exterior to the boundary of the triangle.  Without loss of generality take $x_1 \in L_1, x_2 \in L_2, x_3 \in L_3$ so that $x_2, x_3$ are the perpendicular projections of $x_1$ onto $L_2, L_3$ with $x_1 \not\in \partial Q$.  Then given $L_1$, one may easily choose $L_2, L_3$ such that $d(x_1, x_2) = d(x_1, x_3)$.  We note that because the lines $L_1, L_2, L_3$ are pairwise nonparallel and nonconcurrent, there exists no space $X$ with an infinite number of fixed points.

\begin{lemma}\label{lem2}
For all $x \in X'$, successive preimages of $T$ over $x$ diverge, and successive images of $T$ over $x$ do not diverge.  
\end{lemma}
\begin{proof}
To prove this claim, it is sufficient to show that for any two intersecting lines $L_1, L_2$ in $\mathbb{R}^2$, successive images of $T$ over an initial point $x \in L_1$ converge to the intersection point, while successive preimages from the same point diverge.  Let $L_1, L_2$ intersect at point $z$ with intersection angle $\theta$, $0 < \theta \leq \pi/2$.  Without loss of generality, let $x \in L_1$, so that $T(x) \in L_2$.  Then $\bigtriangleup xT(x)z$ is a right triangle with right angle $\angle xT(x)z$, so that the hypotenuse of $\bigtriangleup xT(x)z$ is $\overline{xz}$.  But because the hypotenuse of the right triangle is $\overline{xz}$, we have that $d(x, z) > d(T(x), z)$, and images of $T$ over $x$ converge to $z$.  It follows that successive images of $T$ do not diverge, but preimages diverge in the space $X$.
\end{proof}

\begin{definition}
Let $U$ label a set of open intervals, where for $I_a \in U$, $I_a \subset X$, and the boundary values of each $I_a$ are given by elements of $\Omega$, so that for all $\omega \in \Omega$, $\omega \not \in I_a$.  
\end{definition}

As such, for $I_k \in U$, we can put $I_k = (\omega_i, \omega_j)$ for adjacent, colinear $\omega_i, \omega_j \in \Omega$.

\begin{lemma}\label{lem7}
For all $x \in I_a$ with $I_a \in U$, $\mathcal{O}(x) \subset \bigcup_{i=0}^m I_i$ for finite $m$.
\end{lemma}
\begin{proof}
This follows immediately from Lemma \ref{lem2}: if successive iterations of $T$ over $x$ do not diverge, then the orbit must be confined in a bounded region, which is necessarily composed of a finite number of intervals because preimages of $T$ from a point diverge, and the boundary values of all intervals are preimages of $T$ from the set of fixed points, and there may only ever be a finite number of fixed points.
\end{proof}

\begin{lemma}\label{lem4}
Let $I_a \in U$.  Then there exists an $I_b \in U$ such that $T(I_a) \subset I_b$.
\end{lemma}
\begin{proof}
We proceed by contradiction and assume $T(I_a) \subset I_b \cup I_c$ so that $I_b \not= I_c$.  By definition, there is no bisector point or preimage of one in the interior of $I_a$, so we know $I_b$ and $I_c$ are adjacent and share a boundary value.  All interval boundary values between adjacent intervals are bisector points or a preimage of $T$ from a bisector point, thus implying that the preimage of the shared boundary point of $I_b$ and $I_c$ is in $I_a$, a contradiction as intervals do not contain preimages of interval boundary points.
\end{proof}

For every triangle $Q$ produced by the lines in $X$, let the interior angles of $Q$ be labelled $\alpha, \beta, \gamma$.  

\begin{definition}
If $I_a \in U$ satisfies the property that for every $x, y \in I_a$ $x \not= y$, then $T^n(x), T^n(y) \in I_a$ so that
\[
n = n(I_a,  \alpha, \beta, \gamma) = \min\{r > 1: T^r(x) \in I_a \}
\]
then we call such intervals \emph{trapping} intervals.
\end{definition}

\begin{lemma}\label{lem5}
In every space $X$ there exists a trapping $I_a \in U$.
\end{lemma}
\begin{proof}
By definition, intervals do not contain fixed points or preimages of fixed points, so iteration from a initial value in an interval will not map to a fixed point in a finite number of steps.  Further, by Lemma  \ref{lem4}, for every $I_a$ there exists an interval $I_b$ such that $T(I_a) \subset I_b$, $I_a \not= I_b$.  Thus, using the fact provided by Lemma \ref{lem7} that for $x \in I_a$, iteration of $T$ over $x$ may only map into $m$ intervals for finite $m$, there must exist an interval $I_i$ such that $T^n(I_i) \subset I_i$.
\end{proof}

\section{Limiting behavior of $T$}\label{sec2}

In this section we show that for $x \in X'$, $T^k(x)$ converges to a periodic orbit as $k \rightarrow \infty$, except for bifurcation points when the map converges to a fixed point, which will be discussed in Section \ref{sec3}.  If iteration of $T$ converges to a fixed point in space $X'$ then we call the space \emph{degenerate}.  Figure \ref{fig6} offers examples of periodic orbits iteration of $T$ generates for different $X$.

\begin{figure}[b]
\centering
\begin{minipage}[t]{5cm}
    \hspace*{.4cm}
	\includegraphics[scale=0.59]{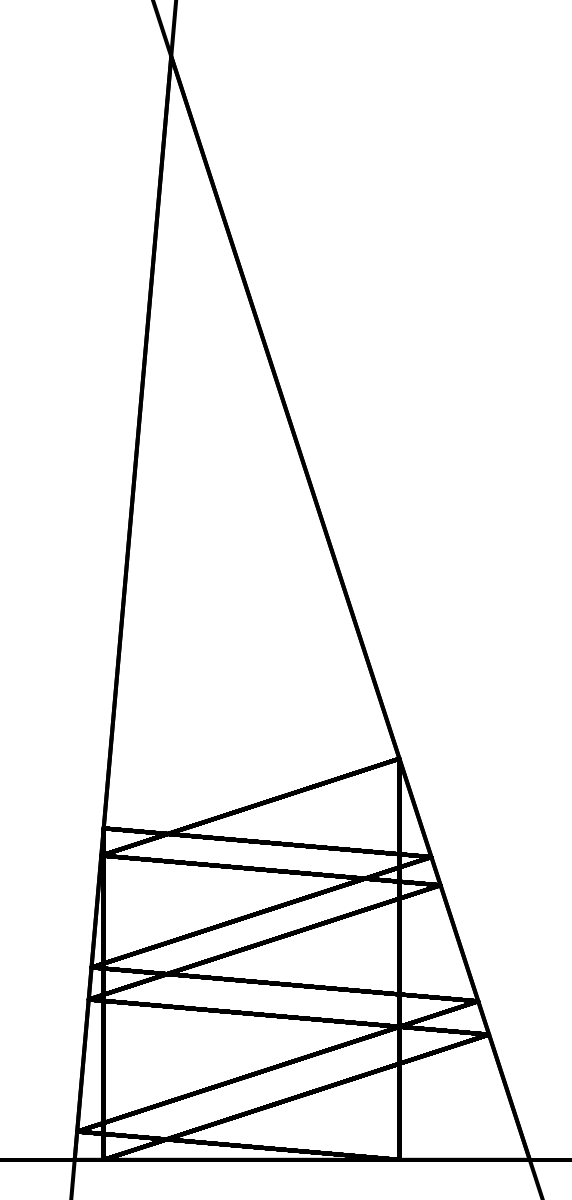}
\end{minipage}
\hspace{3cm}
\begin{minipage}[t]{4cm}
    \hspace*{-4cm}
	\centering
	\includegraphics[scale=0.67]{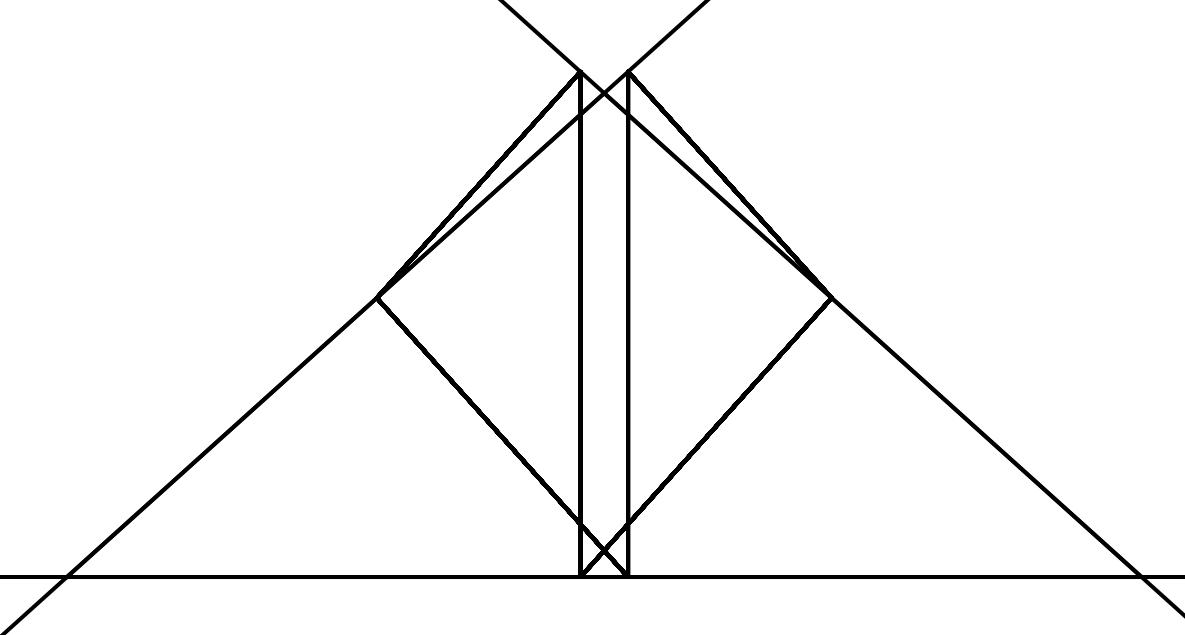}
\end{minipage}
\caption{Periodic orbits generated by $T$ in two different spaces forming scalene triangles.}
\label{fig6}
\end{figure}

\begin{lemma}\label{lem9}
Let $d_s(x, y) = d(T^s(x), T^s(y))$ for $x, y$ in interval $I_a \in U$, and $s = s(I_a, \alpha, \beta, \gamma)$.  Then $d_s(x, y)$ is given by
\begin{equation}\label{eq1}
    d_s(x, y) = \cos^{k_1}(\alpha)\cos^{k_2}(\beta)\cos^{k_3}(\gamma)d(x, y)
\end{equation}
satisfying the condition $k_1 + k_2 + k_3 = s$.  If the triangle $Q = \bigtriangleup ABC$ is acute, then $\alpha, \beta, \gamma$ are the interior angles defining triangle $Q$.  If however $Q$ is obtuse with obtuse angle $\gamma > \pi/2$, replace $\gamma$ with $\gamma' = \pi - \gamma$ in Equation \ref{eq1}.
\end{lemma}
\begin{proof}
Let $x, y \in I_a$, so iterations of $T$ over $x, y$ will follow the same trajectory by Lemma \ref{lem4}, and we proceed by induction on Equation \ref{eq1}, noting that by the geometric behavior of the map, Equation \ref{eq1} will never be negative.  For $s=1$, we have one iteration of the map over $x$ and $y$.  Without loss of generality, let the line segments $\overline{xT(x)}, \overline{yT(y)}$ be opposite an angle with value $\alpha$, so a right triangle is formed and we see that 
\[
d_1(x, y) = d(T(x), T(y)) = \cos(\alpha)d(x, y)
\]

Proceeding to assume Equation \ref{eq1} holds for $s=r$, we show it holds for $s=r+1$.  Take the segments $\overline{T^r(x)T^{r+1}(x)}$ and $\overline{T^r(y)T^{r+1}(y)}$ to be opposite an angle with value $\beta$, so a right triangle is formed and 
\[
d_{r+1}(x, y) = d(T^{r+1}(x), T^{r+1}(y)) = \cos^{k_1}(\alpha)\cos^{k_2 + 1}(\beta)\cos^{k_3}(\gamma)d(x, y)
\]
where $k_1 + k_2 + k_3 = r$.  The reason for the substitution $\gamma' = \pi - \gamma$ if $\gamma >\pi/2$, is that $T$ maps to a perpendicular projection of a point onto another line, and thus the mapping cannot be opposite an obtuse angle. 
\end{proof}

\begin{definition}\label{def2}
Let $I_a \in U$ be a trapping interval in nondegenerate space $X$ such that $\hat{T}:I_a \rightarrow I_a$ is the \emph{induced map} over the interval of continuity $I_a$.  If $x, T^n(x) \in I_a$ for minimal $n$, we let $\hat{T}(x) = T^n(x)$ for all $x \in I_a$.
\end{definition}

\begin{theorem}\label{thm1}
$\hat{T}$ is a contraction mapping over the interval of continuity in the nondegenerate space and has a unique fixed point on the interval.
\end{theorem}
\begin{proof}
Let $x, y \in I_a$ with $I_a$ trapping.  We wish to show the following
\[
d(\hat{T}^r(x), \hat{T}^r(y)) \leq C^r d(x, y)
\]
where $r$ is the $r$th iteration of $\hat{T}$.  We set $C = \text{max}(\cos(\alpha), \cos(\beta), \cos(\gamma))$. As such $C < 1$ is a Lipschitz constant, and by Lemma \ref{lem9} we obtain the following.
\[
d(\hat{T}^r(x), \hat{T}^r(y)) =  (\cos^{k_1}(\alpha)\cos^{k_2}(\beta)\cos^{k_3}(\gamma))^rd(x, y) 
= \phi^r d(x, y)
\]
where $k_1 + k_2 +k_3= n$ with $n$ being the $n$th iterate given by $\hat{T}(x) = T^n(x)$.  We have $\phi^r \leq C^r$ so $\hat{T}$ is a contraction mapping with unique fixed point in $I_a$.  
\end{proof}

\begin{cor}\label{cor1}
Let $I_a$ be a trapping interval in any non-degenerate space with $x, T^n(x) \in I_a$ so that $n$ is minimal.  Then $T^k(x)$ will converge to a periodic orbit of period $n$ as $k \rightarrow \infty$.
\end{cor}
\begin{proof}
By Theorem \ref{thm1}, $\hat{T}$ is a contraction mapping so $\hat{T}^r(x_0)$ converges to a fixed point as $r \rightarrow \infty$ for $x_0$ in trapping $I_a$.  It then follows that the fixed point of $\hat{T}$ on the trapping interval is a periodic point for $T$, and the corresponding periodic orbit has period $n$ by definition of the induced map given in Definition \ref{def2}.
\end{proof}

\begin{theorem}\label{thm2}
For all $x \in X'$ for non-degenerate $X'$, $T^k(x)$ will converge to a periodic orbit as $k \rightarrow \infty$.
\end{theorem}
\begin{proof}
By consequence of Lemma \ref{lem7} and proof of Lemma \ref{lem5}, we have that for all $x \in X'$, there exists a $T^k(x) \in I_a$ for trapping $I_a$ and finite $k$.  But if $T^k(x) \in I_a$ for trapping $I_a$, then by Corollary \ref{cor1} the map must converge to a periodic orbit.
\end{proof}

Figure \ref{fig2} provides visual demonstration of convergence to a periodic orbit as proved in Theorem \ref{thm2}.  The red lines indicate iterations of the map converging to the orbit, and the black path traces the periodic orbit.

\begin{figure}
    \centering
    \includegraphics[scale=.75]{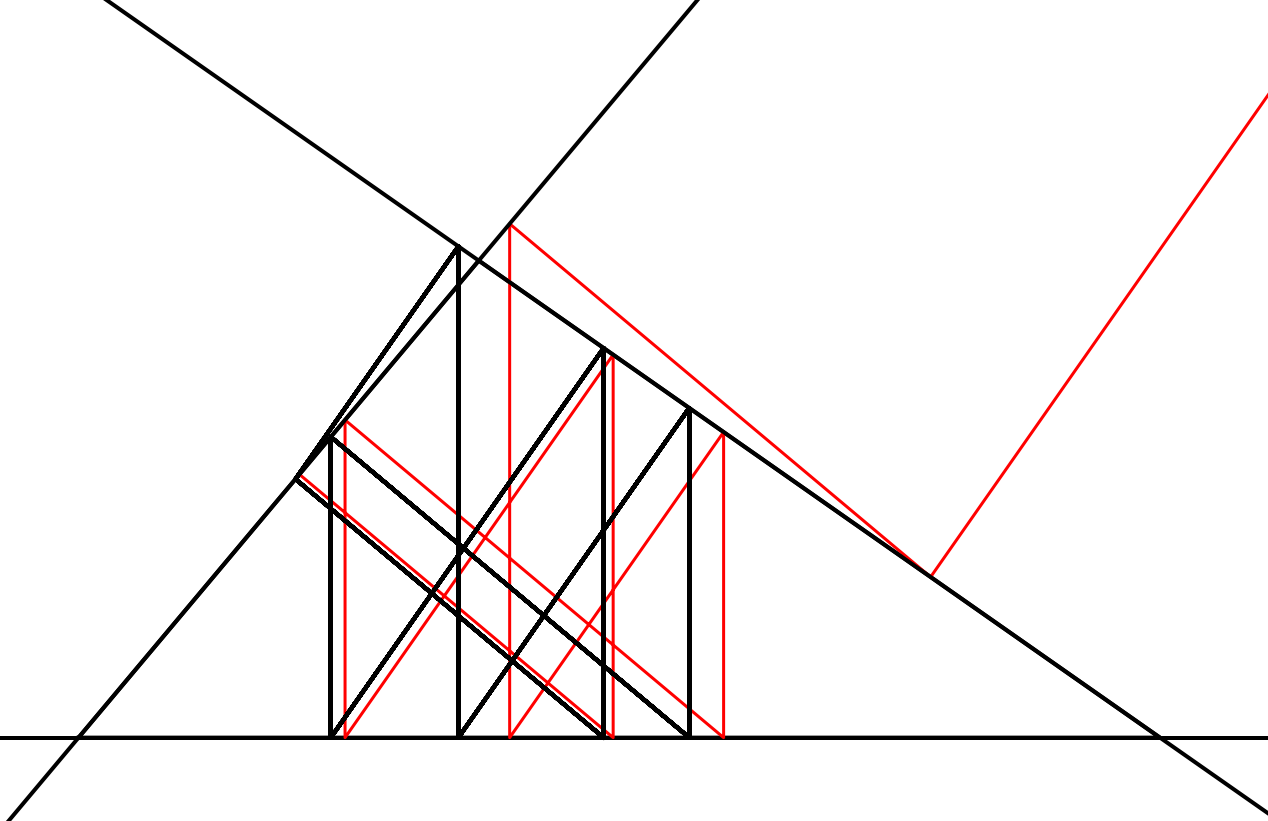}
    \caption{}
    \label{fig2}
\end{figure}

\section{Periodic orbits in isosceles triangles}\label{sec3}

For the purpose of this section we restrict the orbit of $T$ to the boundary of the triangle so that $T:\partial Q \rightarrow \partial Q$, and we take every initial point of the map to be in a trapping interval.  As such, in the case when $Q$ is obtuse and there only exists one possible perpendicular projection from a particular point to an edge, let $T$ map to that edge.  That trapping intervals still exist is clear: such restriction serves to disregard edge cases unimportant to this analysis when the map has periodic points outside of the triangle boundary.  Let triangle $Q$ be isosceles with interior angles $\alpha, \beta, \gamma$ so that $\beta =\gamma = \frac{\pi-\alpha}{2}$, and we let $\alpha$ be an adjustable parameter.

\begin{lemma}\label{lem6}
If $0 < \alpha < \pi$ is the unique angle of an isosceles triangle $Q$ over which $T$ generates a periodic orbit of period $n$, then there must exist some $\mu \in (0, \alpha)$ so the triangle with unique angle $\alpha - \mu$ admits a periodic orbit with period $n' \not= n$.
\end{lemma}
\begin{proof}
Let the triangle \emph{base} be the edge opposite $\alpha$.  We proceed by contradiction and assume the orbit with period $n$ in the isosceles triangle determined by $\alpha$ exists in all isosceles triangles with unique angle $\alpha - \mu < \alpha$.  Increase in value of $\mu$ corresponds with increase in value of $\beta, \gamma$, so that the bisector points of $\beta$ and $\gamma$ shift away from the triangle base.  With this fact we obtain our contradiction, for there must exist some sufficiently large $\mu$ such that a bisector point and periodic point switch orientation on the edge, thereby forcing the addition of at least one more step in the iteration of the map before return to a trapping interval, implying a new orbit with distinct period.
\end{proof}

Immediate from Lemma \ref{lem6} we obtain the following.

\begin{prop}\label{prop3}
For each distinct isosceles triangle with parameter $\alpha$ where $0 < \alpha < \pi$, there exists an open interval denoted $(\alpha_{j}, \alpha_{j-1})$, where if $\alpha \in (\alpha_{j}, \alpha_{j-1})$ then iteration of $T$ must map to a particular set of periodic orbits with period unique to the interval.
\end{prop}

The boundary values of such intervals $(\alpha_{j}, \alpha_{j-1})$ as described in Proposition \ref{prop3} are bifurcation points for the system. 

\begin{lemma}\label{lem10}
When $\alpha = \alpha_j$, a bifurcation point, then $T^k(x)$ converges to a bisector point as $k \rightarrow \infty$.
\end{lemma}
\begin{proof}
Period change of a particular orbit indicates inflection of a periodic point over a bisector point, which must have occurred at some boundary value of an $(\alpha_{j}, \alpha_{j-1})$ interval.  So when $\alpha$ is equal to a boundary value of a such an interval the orbit must converge to a bisector point, which is a fixed point.
\end{proof}

As a consequence of Lemma \ref{lem10}, we have that trapping intervals in a degenerate space may contain one of their boundary points should that boundary point be a bisector point, so that the unique fixed point the induced map converges to is a bisector point.  As such, with Lemma \ref{lem10} we may establish an immediate corollary to Theorem \ref{thm2}.

\begin{cor}
For all $x \in X$, $T^k(x)$ will converge to a fixed point or periodic orbit as $k \rightarrow \infty$, or iteration of $T$ will map to a fixed point in a finite number of iterations.
\end{cor}

\begin{figure}
    \hspace*{-1.9cm} 
    \includegraphics[scale=.26]{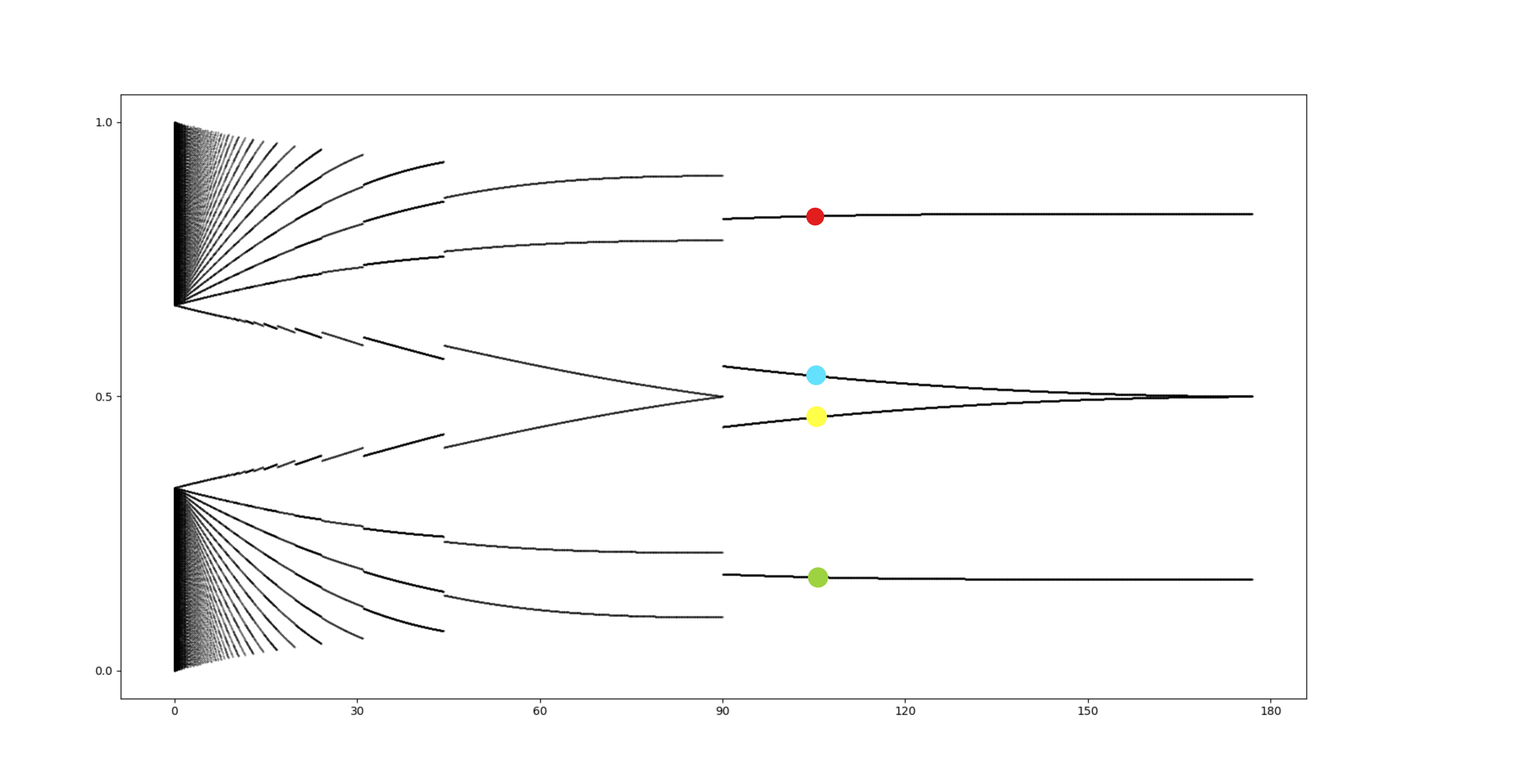}
    \caption{Bifurcation diagram, change of $\alpha$ is in degrees.}
    \label{fig8}
\end{figure}

\begin{figure}
    \centering
    \includegraphics[scale=.08]{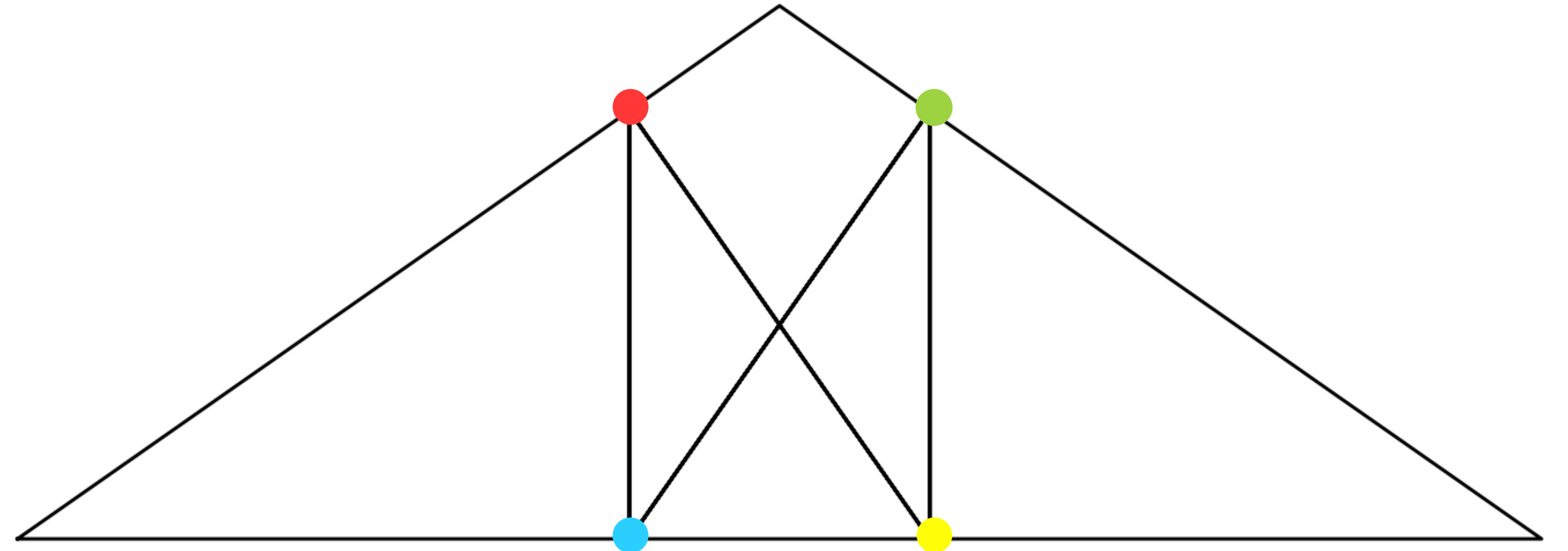}
    \caption{Orbit with colored periodic points corresponding to the  bifurcation diagram.}
    \label{fig9}
\end{figure}

\begin{remark}
For a particular isosceles triangle $Q$, if iterations of the map $T$ converge to an orbit with odd period, then the periodic orbit has two possible orientations in $Q$.
\end{remark}

Figure \ref{fig8} gives a bifurcation diagram for the application of $T$ to isosceles triangles for change of the parameter $\alpha$ along the $x$ axis.  Figure \ref{fig9} offers explicit demonstration of how the periodic points of each orbit were graphed through ``unfolding" the triangle and plotting the periodic points. The $y$ axis represents triangle perimeter distance between periodic points, where the length of the triangle base is fixed to $1/3$.  The bifurcation diagram plots the periodic points of both possible orientations of odd period orbits.

The dynamics of the map $T$ applied to the space of three nonconcurrent, pairwise nonparallel lines leads to questions concerning the dynamics of the map when generalized to different spaces such as non-Euclidean planar triangles.  Would iteration of the map over the boundary of a non-Euclidean, non-convex triangle generate periodic orbits, or would the dynamical system be topologically transitive over a particular subset of the space?  In fact, we may also ask how the asymptotic behavior of the map would change when applied to $n > 2$ dimensional convex polytopes, where rather than mapping to ``farthest" lines, the map iterates to ``farthest" faces, normal to the face.

\end{document}